\documentclass[12pt,A4,leqno]{amsart}
\usepackage{amsfonts,bm}
\usepackage{mathrsfs}
\usepackage[T1]{fontenc}
\usepackage{amssymb}
\usepackage{amsmath,amscd}
\usepackage{color}
\usepackage{epsf}
\usepackage{graphicx}
\usepackage{xypic}
\usepackage{extarrows}
\usepackage{hyperref} 
\hypersetup{colorlinks,linkcolor={blue},citecolor={blue}} 
\usepackage[capitalise]{cleveref}

\theoremstyle{plain}
\newtheorem{thm}{Theorem}[section]
\newtheorem{lem}[thm]{Lemma}
\newtheorem{prop}[thm]{Proposition}
\newtheorem{cor}[thm]{Corollary}

\theoremstyle{definition}
\newtheorem{defi}[thm]{Definition}
\newtheorem{rem}[thm]{Remark}

\newcommand{\R}{\mathbb R}
\newcommand{\Z}{\mathbb Z}
\newcommand{\F}{\mathbb F}
\newcommand{\nn}{\vskip 0.2cm}
\newcommand{\n}{\vskip 0.1cm}

\newcommand{\FF}{\mathcal F}
\newcommand{\ZZ}{\mathcal Z}

\renewcommand{\geq}{\geqslant}
\renewcommand{\leq}{\leqslant}

\newcommand{\wt}{\widetilde}
\renewcommand{\subset}{\subseteq}

\begin{document}

\title [\ ] {On Simplicial Complexes with Extremal 
Total Betti number and Total Bigraded Betti Number}

\author{Pimeng Dai*, \  Li Yu**}
\address{*School of Mathematics, Nanjing University, Nanjing, 210093,
P.R.China}
 \email{pimengdai@gmail.com}
\address{**School of Mathematics, Nanjing University, Nanjing, 210093,
P.R.China}
 \email{yuli@nju.edu.cn}




\begin{abstract}
We determine which simplicial complexes have
the maximum or minimum sum of Betti numbers and sum of bigraded Betti numbers with a given number of vertices
in each dimension. 
 \end{abstract}

\maketitle

 \section{Introduction}
  For a finite CW-complex $X$ and a field $\F$, 
  let
\[ \beta_i(X;\F) := \dim_{\F} H_i(X;\F), \ \ \ \widetilde{\beta}_i(X;\F) := \dim_{\F} \widetilde{H}_i(X;\F)\]  
   where $H_i(X;\F)$ and $\widetilde{H}_i(X;\F)$ are the singular and reduced singular homology group of $X$ with $\F$-coefficients, respectively. Moreover, 
  define
   \[ tb(X;\F):= \sum_{i} \beta_i(X;\F), \ \ \  \widetilde{tb}(X;\F):= \sum_{i} \widetilde{\beta}_i(X;\F). \]  
    We call $tb(X;\F)$ and $\widetilde{tb}(X;\F)$ the  \emph{total Betti number} and \emph{reduced total Betti number} of $X$ with $\F$-coefficients, respectively. 
    The difference between $tb(X;\F)$ and $\widetilde{tb}(X;\F)$ is just $1$. 
 In addition, we call
 \[  \widetilde{\chi}(X):= \sum_{i} (-1)^i \widetilde{\beta}_i(X;\F) \]
  the \emph{reduced Euler characteristic} of $X$, which is independent on
 the coefficient $\F$.\n
    
    \noindent \textbf{Convention:} The coefficient field $\mathbb{F}$ will be omitted when there is no ambiguity in the context or the coefficients are not essential. In fact, all the results obtained
    in this paper are independent on the coefficients.\n
    
     The total Betti number of a topological space plays
     an important role in many theories in mathematics. For example, in
    the Arnold conjecture in symplectic geometry, the total Betti number of a symplectic manifold $M$ serves as the lower bound of
    the number of fixed points of any
     Hamiltonian symplectomorphism of $M$; see Abbondandolo~\cite[Chapter 6]{Abb01} for more information.
    Another example is from the Smith theory of cyclic group actions where the Smith inequality says that for a prime $p$ and  a $\Z_p$-action
    on a finite CW-complex $X$,     
    the fixed point set $X^{\Z_p}$ must satisfy 
   $tb(X^{\Z_p};\Z_p) \leq tb(X;\Z_p)$; see Allday and Puppe~\cite[Chapter 1]{AllPuppe93}. In addition,
  an important conjecture in equivariant topology 
  and rational homotopy theory
  due to Halperin~\cite{Halp85} states that if a
   $k$-dimensional torus $T^k=(S^1)^k$ can act almost freely on a finite-dimensional topological space $X$, then $tb(X;\mathbb{Q})\geq 2^k$. A parallel conjecture due to Carlsson~\cite{Carl86} states that for a prime $p$, if a $p$-torus $(\Z_p)^k$ 
  can act freely on a finite CW-complex $X$, then $tb(X;\Z_p)\geq 2^k$. \n
   
    In this paper, we focus our
    study on the total Betti number and total bigraded Betti number (see Definition~\ref{Def:Bigraded-Betti}) of a simplicial complex with a fixed number of vertices. 
     Let $K$ be a simplicial complex whose vertex set is
    $$\mathrm{Ver}(K)=[m]=\{1,2,\cdots, m\}.$$
    
     Each simplex
    $\sigma$ of $K$ is considered as a subset of $[m]$.
     Let $\mathrm{Star}_K \sigma$ and $\mathrm{Link}_K \sigma$ denote
    the \emph{star} and the \emph{link} of $\sigma$ in $K$, respectively. 
    For any subset $J\subseteq [m]$, let 
    \[ K|_J = \ \text{the \emph{full subcomplex} of $K$ obtained by restricting to}\ J.\]  
   In particular, when $J=\varnothing$, $K|_J=\varnothing$ and
   define
  \[ \beta_{i}(\varnothing)=0, \ \forall i\geqslant 0; \ \ \ \widetilde{\beta}_{i}(\varnothing)=  \begin{cases}
   1 ,  &  \text{if $i= -1$ }; \\
   0,  &  \text{otherwise}.
 \end{cases} \]

\textbf{Question 1:} For a positive integer $m$, 
 which simplicial complexes have the
  maximum (reduced) total Betti number among
  all the simplicial complexes with $m$ vertices?
 \n
 When $m=1,2,3$,  the answer is just the discrete $m$ points.
 When $m=4$, the answer is either the discrete $4$ points or the complete graph on $4$ vertices.
 A complete answer to Question 1 is contained in the following theorem. \n

 \begin{thm}[\text{Bj\"{o}rner and 
 Kalai~\cite[Theorem 1.4]{BjrnKal88}}] \label{Thm-BjrnKal}
  Suppose $K$ is a simplicial complex with at most $n+1$ vertices. Then
$
|\widetilde{\chi}(K)| \leqslant \widetilde{tb}(K) \leqslant
 \binom{n}{[n / 2]}
$. 
Moreover, the following conditions are equivalent:
\begin{itemize}
\item[(i)] $|\widetilde{\chi}(K)|= \binom{n}{[n / 2]}$,
\item[(ii)] $\widetilde{tb}(K)= \binom{n}{[n / 2]}$,
\item[(iii)] $K$ is the $k$-skeleton of an $n$-simplex, where $k=n / 2-1$ if $n$ is even and $k=(n-1) / 2$ or $k=(n-3) / 2$ if $n$ is odd.
\end{itemize}
 \end{thm}
 
  Note that Theorem~\ref{Thm-BjrnKal} not only gives the answer to Question 1, but also reveal the connection between $\widetilde{tb}(K)$ and the Euler characteristic of $K$. The proof of Theorem~\ref{Thm-BjrnKal} in~\cite{BjrnKal88} uses a nontrivial operation called \emph{algebraic shifting} of a simplicial complex and a well-known theorem of Sperner in~\cite{Spern28}. But if we only want to answer Question 1, we can use some easier argument as shown in Section~\ref{Sec-TB}. \n

  Moreover, we can ask a
  more subtle question than Question 1 as follows.\n
  
  \textbf{Question 2:} For each $0\leqslant d < m$,
 which $d$-dimensional simplicial complexes with $m$ vertices have the maximum (reduced) total Betti number among
 all the $d$-dimensional simplicial complexes with $m$ vertices?
 \n
 
 For a pair of integers $(m, d)$, $0 \leqslant d < m$, we introduce the following notations:
 
 \begin{itemize}
 \item  Let $\Sigma(m)$ denote the set of all simplicial complexes with vertex set $[m]$.\n
   
 \item Let $\Sigma(m, d)$ denote the set of all 
 $d$-dimensional simplicial complexes with vertex set $[m]$.
 \end{itemize}

 Moreover, define (with respect to a fixed coefficient field $\mathbb{F}$)
	$$\Sigma^{tb}(m)=\Big\{ K\in\Sigma(m) \mid \widetilde{tb}(K)=\max_{L\in\Sigma(m)}\widetilde{tb}(L)\Big\}\subseteq \Sigma(m).$$  
$$
\Sigma^{t b}(m, d)=\Big\{K \in \Sigma(m, d) \mid \widetilde{tb}(K)=\max _{L \in \Sigma(m, d)} \widetilde{tb}(L) \Big\} \subseteq \Sigma(m, d) .
$$

 We consider $\Sigma(m)$, $\Sigma(m, d)$, $\Sigma^{tb}(m)$ and $
\Sigma^{tb}(m, d)$ as partial ordered sets with respect to the inclusion of simplicial complexes. Clearly,
if $K$ is an (maximal or minimal) element of $\Sigma^{tb}(m)$, then $K$ is also an (maximal or minimal) element of $\Sigma^{tb}(m,\dim(K))$.
\n

In addition, for any $m\geq 1$, we use $\Delta^{[m]}$ to denote the $(m-1)$-dimensional simplex with vertex set $[m]$. So the boundary $\partial \Delta^{[m]}$ of $\Delta^{[m]}$ is a simplicial sphere of dimension $m-2$. Moreover, for
any $0 \leqslant k < d < m$,
\begin{itemize}
\item let
$\Delta^{[m]}_{(k)}$ denote the $k$-skeleton of $\Delta^{[m]}$;\n
\item let $\Delta_{(k)}^{[m]}\langle d\rangle$ denote the minimal $d$-dimensional subcomplex of $\Delta^{[m]}$ that contains $\Delta_{(k)}^{[m]}$, which is unique up to simplicial isomorphism. Indeed, we can think of
 $\Delta_{(k)}^{[m]}\langle d\rangle$ as the union of
 $\Delta_{(k)}^{[m]}$ with the $d$-simplex $\Delta^{[d+1]}$ in $\Delta^{[m]}$ where $[d+1]=\{1,\cdots,d+1\}\subseteq [m]$.
\end{itemize}
\n

We will prove the following theorem in Section~\ref{Sec-TB} which answers Question 2.

\begin{thm}[see Theorem~\ref{Thm-Main-2}] 
 The sets $\Sigma^{tb}(m,d)$ are classified as follows:
 \begin{itemize}
  \item[(i)] If $d \leqslant\left[\frac{m}{2}\right]-1$ or $d=m-1$, then $\Sigma^{tb}(m,d)=\Big\{\Delta_{(d)}^{[m]}\Big\}$;
  \item[(ii)] If $\left[\frac{m}{2}\right] \leqslant d \leqslant m-3 $, then $\Sigma^{tb}(m,d)=
  \Big\{\Delta_{\left(\left[\frac{m}{2}\right]-1\right)}^{[m]}\langle d\rangle \Big\}$;
  \item[(iii)] If $d=m-2$,
    \begin{itemize}
      \item when $m$ is odd, $\Sigma^{tb}(m,d)=\Big\{\Delta_{\left(\left[\frac{m}{2}\right]-1\right)}^{[m]}\langle d\rangle, \Delta_{\left(\left[\frac{m}{2}\right]\right)}^{[m]}\langle d\rangle \Big\}$;
      
 \item when $m$ is even, $\Sigma^{tb}(m,d)=\Big\{\Delta_{\left(\left[\frac{m}{2}\right]-1\right)}^{[m]}\langle d\rangle\Big\}$.
  \end{itemize}
  \end{itemize}
\end{thm}
 
 By the above Theorem, the set $\Sigma^{tb}(m,d)$
 is independent on the coefficients $\mathbb{F}$.

\begin{rem}
It is also meaningful to ask what kind of simplicial complexes with $m$ vertices have the minimum total Betti number. But the answer is just all the acyclic simplicial complexes which are numerous; see Kalai~\cite{Kal85} for more discussion on this subject.
\end{rem}

  There is another family of integers associated
  to a simplicial complex $K$ called bigraded Betti numbers, which are derived from the Stanley-Reisner ring of $K$. Recall the \emph{Stanley-Reisner ring} of $K$ over 
$\mathbb{F}$ (see Stanley~\cite{Stanley07}) is 
\[ \mathbb{F}[K]= \mathbb{F}[v_1,\cdots,v_m] \slash \mathcal{I}_K\]
where $\mathcal{I}_K$ is the ideal in the polynomial ring $\mathbb{F}[v_1,\cdots,v_m]$ generated by all the square-free monomials
$v_{i_1}\cdots v_{i_s}$ where $\{i_1,\cdots , i_s\}$ is not a simplex of $K$. Since $\mathbb{F}[K]$ is naturally a module over
$\mathbb{F}[v_1,\cdots,v_m]$, by the standard construction in
homological algebra, we obtain a canonical 
algebra $\mathrm{Tor}_{\mathbb{F}[v_1,\cdots, v_m]}(\mathbb{F}[K],\mathbb{F})$ from $\mathbb{F}[K]$, where $\mathbb{F}$
is considered as the trivial $\mathbb{F}[v_1,\cdots, v_m]$-module.
Moreover, there is a bigraded $\mathbb{F}[v_1,\cdots,v_m]$-module structure on $\mathrm{Tor}_{\mathbb{F}[v_1,\cdots, v_m]}(\mathbb{F}[K],\mathbb{F})$ (see~\cite{Stanley07}):   
  \[ \mathrm{Tor}_{\mathbb{F}[v_1,\cdots, v_m]}(\mathbb{F}[K],\mathbb{F}) = \bigoplus_{i,j\geq 0} \mathrm{Tor}^{-i,2j}_{\mathbb{F}[v_1,\cdots, v_m]}(\mathbb{F}[K],\mathbb{F}) \] 
  where $\mathrm{deg}(v_i)=2$ for each $1\leqslant i \leqslant m$. The integers 
  $$\beta^{-i,2j}(\mathbb{F}(K)) := \dim_{\mathbb{F}}\mathrm{Tor}^{-i,2j}_{\mathbb{F}[v_1,\cdots, v_m]}(\mathbb{F}[K],\mathbb{F})$$
  are called
the \emph{bigraded Betti numbers of $K$ with $\mathbb{F}$-coefficients}.      
   Note that unlike Betti numbers of $K$, bigraded Betti numbers are not topological invariants, but only combinatorial invariants of $K$ in general.
     
     \begin{defi} \label{Def:Bigraded-Betti}
   The \emph{total bigraded Betti number} of $K$ with $\mathbb{F}$-coefficients is
       \begin{equation*}
         \widetilde{D}(K;\mathbb{F}) = \sum_{i,j} \beta^{-i,2j}(\mathbb{F}(K)) = \dim_{\mathbb{F}} \mathrm{Tor}_{\mathbb{F}[v_1,\cdots, v_m]}(\mathbb{F}[K],\mathbb{F}). 
         \end{equation*}
     \end{defi}
 
By Hochster's formula (see Hochster~\cite{Hoc77} or~\cite[Theorem 4.8]{Stanley07}), the bigraded Betti numbers of $K$ can also be computed from the homology groups of the full subcomplexes of $K$:
 \begin{equation}\label{Equ:Hochster}
   \beta^{-i,2j}(\mathbb{F}(K)) = \sum_{J\subseteq [m], |J|=j} 
   \dim_{\mathbb{F}} \widetilde{H}_{j-i-1}(K|_J;\mathbb{F}).
 \end{equation} 
  
 So we can also express $\widetilde{D}(K;\mathbb{F})$ as
    \begin{equation} \label{Equ:D-tilde-K-J}
    \widetilde{D}(K;\mathbb{F})= \sum_{J\subseteq [m]} \widetilde{tb}(K|_J;\mathbb{F}). 
    \end{equation}
    This explains why we put ``$\, \widetilde{\ }\, $'' in the notation
    $\widetilde{D}(K;\mathbb{F})$.\n
 
  \textbf{Question 3:} For each $0\leqslant d < m$, which $d$-dimensional simplicial complexes with $m$ vertices have the minimum total 
  bigraded Betti number over a field $\mathbb{F}$ among all the $d$-dimensional simplicial complexes with $m$ vertices?  
  We readily call such kind of simplicial complexes
  \emph{$\widetilde{D}$-minimal} (over $\mathbb{F}$).\n
  
 It follows from Cao and L\"u~\cite[Theorem 1.4]{CaoLu12} or Ustinovsky~\cite[Theorem 3.2]{Uto11} that
 there is a universal lower bound of $\widetilde{D}(K;\mathbb{F})$ (see~\eqref{Equ:Uni-Lower-bound}):
  \begin{equation*} 
     \widetilde{D}(K;\mathbb{F}) \geqslant 2^{m-d-1},
     \, K\in \Sigma(m,d).
     \end{equation*}

 \begin{defi}[Tight Simplicial Complex] \label{Def:Tight-Complex}
   Let $K$ be a $d$-dimensional simplicial complex with $m$ vertices. We call $K$ \emph{tight} (over $\mathbb{F}$)
   if $\widetilde{D}(K;\mathbb{F}) = 2^{m-d-1}$.
   \end{defi}
  
 We classify all the tight simplicial complexes in the following theorem.
  
   \begin{thm}[see Theorem~\ref{Thm-Main-3}] \label{Thm:Main-3}
    A finite simplicial complex
     $K$ is tight if and only if $K$ is
     of the form $ \partial \Delta^{[n_1]}*\cdots*\partial \Delta^{[n_k]}$ or $\Delta^{[r]} * \partial \Delta^{[n_1]}*\cdots*\partial \Delta^{[n_k]}$ for some positive integers
     $n_1,\cdots, n_k$ and $r$.
  \end{thm}
  
 It is a convention to let $\partial\Delta^{[1]}=\varnothing$ and
 $K*\varnothing = K$. Theorem~\ref{Thm:Main-3} implies that the tightness of a simplicial complex is independent on the coefficient field
 $\mathbb{F}$. \n
 
 From Theorem~\ref{Thm:Main-3}, we can easily deduce that if $K\in \Sigma(m,d)$ is tight, it is necessary that $\left[ \frac{m-1}{2}\right]\leq d$. In particular, the equality
 $\left[ \frac{m-1}{2}\right]= d$ is achieved by $\partial \Delta^{[2]}*\partial \Delta^{[2]} *\cdots*\partial \Delta^{[2]}$ when $m$ is even and by $\Delta^{[1]}*\partial \Delta^{[2]}*\partial \Delta^{[2]} *\cdots*\partial \Delta^{[2]}$ when $m$ is odd. Conversely, for any $(m,d)$ 
 with $\left[ \frac{m-1}{2}\right]\leq d \leq m-1$, there always exists
 a tight simplicial complex $K\in \Sigma(m,d)$.
 So if $\left[ \frac{m-1}{2}\right] \leq d \leq m-1$,
 the $\widetilde{D}$-minimal simplicial complexes in
 $\Sigma(m,d)$ are exactly all the tight simplicial complexes.
   \n
   
   But when $\left[ \frac{m-1}{2}\right] > d$, a $\widetilde{D}$-minimal simplicial complex in $\Sigma(m,d)$ is never tight, and it seems to us that there is no very good way to describe $\widetilde{D}$-minimal simplicial complexes in general. One reason is that
the full subcomplexes of a $\widetilde{D}$-minimal simplicial complex may not be $\widetilde{D}$-minimal. For example, by exhausting all the $33$ members of
$\Sigma(5,1)$, we find that all the $\widetilde{D}$-minimal simplicial complexes in $\Sigma(5,1)$ are $K_{2,3}$ and $C_5$ (see Figure~\ref{p:K1K2}) whose $\widetilde{D}$-value is $12$.  Note that none of the full subcomplexes of $C_5$ on four vertices are $\widetilde{D}$-minimal. 
 \begin{figure}[h]
         \includegraphics[width=0.46\textwidth]{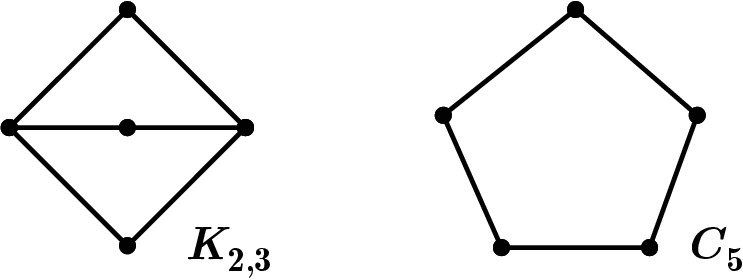}\\
          \caption{$\widetilde{D}$-minimal $1$-dimensional simplicial complexes with $5$ vertices}\label{p:K1K2}
      \end{figure} 
      
      This example suggests that we may not be able to inductively construct all the $\widetilde{D}$-minimal simplicial complexes
with $m$ vertices from the $\widetilde{D}$-minimal ones with $m-1$ vertices. In addition, it is not clear how to
compute the minimal value of $\widetilde{D}(\cdot)$
on $\Sigma(m,d)$ when $\left[ \frac{m-1}{2}\right] > d$ except exhausting all the members.\n

   Moreover, we can ask the following question 
   parallel to Question 1 for bigraded Betti numbers.\n
    
  \textbf{Question 4:} For a positive integer $m$, which simplicial complexes with $m$ vertices have the
  maximum total bigraded Betti numbers among
  all the simplicial complexes with $m$ vertices?\nn
  
  In Section~\ref{Sec-Maximum-D}, we prove the following theorem which answers Question 4.
 Let
  \begin{equation} \label{Equ:g-def}
    g(m,d)=\sum_{j=d+1}^{m}\binom{m}{j}\binom{j-1}{d},\ 0\leq d< m.
    \end{equation}
  
  \begin{thm}[Theorem~\ref{Thm-Main-4}]
   If $K$ is a simplicial complex with $m$ vertices, then
$$
\widetilde{D}(K;\mathbb{F}) \leqslant g\left(m,\left[\frac{m-1}{3}\right]\right)+1
$$
 for any field $\mathbb{F}$, where the equality holds if and only if 
$K=\Delta_{\left(\left[\frac{m-1}{3}\right]-1\right)}^{[m]}$.
\end{thm}

 \begin{rem}
  Parallelly to Question 2, we can also ask for each
  $0\leqslant d < m$,
 which simplicial complexes $K\in \Sigma(m,d)$ have the maximum total bigraded Betti number over a field $\mathbb{F}$ among 
 all the members of $\Sigma(m,d)$. But we do not know the complete answer of this question yet. The question can be turned into a very complicated combinatorial extremum problem.
 \end{rem}
 
 The paper is organized as follows. In Section~\ref{Sec-TB}, we first write an easy argument using Mayer-Vietoris sequence to answer Question 1. Then we use the theory of algebraic shifting of simplicial complexes from~\cite{BjrnKal88} and some results on Sperner families from~\cite{Lih80} and~\cite{Gri82} to give a complete answer to Question 2.
 In Section~\ref{Sec-Minimum-D}, we classify all the tight simplicial complexes using some results from~\cite{Uto11} and~\cite{YuMas21}. In Section~\ref{Sec-Maximum-D}, we give a complete answer to Question 4 via some combinatorial inequality
 proved in the appendix.
 
\vskip .4cm
  
   \section{Simplicial complexes with the maximum total Betti number} \label{Sec-TB}
 
 In this section, we first give an alternative proof of Theorem~\ref{Thm-Main-1} which answers Question 1.  Our approach is different from~\cite{BjrnKal88}.
 
 \begin{lem}\label{lem:mv}
For any two subcomplexes $K,L$ of $\Delta^{[m]}$, one always have
	\begin{equation} \label{Equ:tb-equ}
	 \widetilde{tb}(K)+\widetilde{tb}(L) \leq \widetilde{tb}(K \cap L)+ \widetilde{tb}(K \cup L).
	 \end{equation}
\end{lem}
\begin{proof}
 By the following Mayer-Vietoris sequence for $(K, L)$ :  
$$\cdots \widetilde{H}_{j+1}(K \cup L) \overset{d_{j+1}}{\longrightarrow} \widetilde{H}_j(K \cap L) \longrightarrow \widetilde{H}_j(K) \oplus \widetilde{H}_j(L) \longrightarrow \widetilde{H}_j(K \cup L) \overset{d_j}{\longrightarrow} \widetilde{H}_{j-1}(K \cap L)  \cdots,$$ 
we obtain:
$$\widetilde{\beta}_j(K)+\widetilde{\beta}_j(L)\leq\widetilde{\beta}_j(K \cap L)+\widetilde{\beta}_j(K \cup L), \ \forall j\geq 0.$$
Then sum it over all $j\geq 0$, we get the desired inequality.
\end{proof}
 
 Note that the equality in~\eqref{Equ:tb-equ} holds if and only if 
$\mathrm{Im}(d_j)=0$ for all $j\geqslant 0$ in the above
Mayer-Vietoris sequence.

  \begin{thm} \label{Thm-Main-1}
  Let $K$ be a simplicial complex with $m$ vertices. Then
 $$ \widetilde{tb}(K) \leqslant \binom{m-1}{\left[\frac{m-1}{2}\right]} .$$ 
Moreover,  $\widetilde{tb}(K)= \binom{m-1}{\left[\frac{m-1}{2}\right]}$ if and only if $K$ is isomorphic to the $k$-skeleton of an 
$(m-1)$-dimensional simplex, where $k=\frac{m-3}{2}$ if $m$ is odd and $k=\frac{m}{2}-1$ or $k=\frac{m}{2}-2$ if $m$ is even.
 \end{thm}
  \begin{proof}
   Suppose the vertex set of $K$ is 
   $[m]=\{1,\cdots, m\}$ and consider $K$ as a subcomplex of $\Delta^{[m]}$. In the rest of the proof, we assume that $m\geqslant 3$.\n

   Note that each permutation $\phi$ of $[m]$ determines a simplicial isomorphism $\Delta^{[m]}\rightarrow \Delta^{[m]}$, still denoted by $\phi$.
Let $\phi K$ denote the image of $K$ under $\phi$. Then by Lemma~\ref{lem:mv}, 
\begin{equation} \label{Equ:tb-K}
\widetilde{tb}(K)+\widetilde{tb}(\phi K) \leqslant \widetilde{tb}(K \cap \phi K)+ \widetilde{tb}(K \cup \phi K).
\end{equation}

 \n

 Since $\phi K$ is homeomorphic to $K$,
 if $K$ is an element of $\Sigma^{tb}(m)$, then both $K \cap \phi K$ and $K \cup \phi K$ must also belong to $\Sigma^{tb}(m)$ since
 their vertex sets are also $[m]$. So in particular, if $K$ is a maximal (or minimal) element of $\Sigma^{tb}(m)$, we must have $K= K \cup \phi K$ (or $K=K\cap \phi K$) for every permutation $\phi$ of $[m]$, which implies $K=\phi K$. Then
 $K$ is the skeleton $\Delta^{[m]}_{(\dim(K))}$.
    An easy calculation shows  
     $$ \widetilde{tb}\big(\Delta^{[m]}_{(d)}\big) = \binom{m-1}{d+1}.$$
    So $\widetilde{tb}\big(\Delta^{[m]}_{(d)}\big)$  reaches the maximum $\binom{m-1}{\left[\frac{m-1}{2}\right]} $ when and only when
    $$d=  \begin{cases}
   \frac{m-3}{2} ,  &  \text{if $m$ is odd}; \\
   \frac{m}{2}-1\ \text{or}\ \frac{m}{2}-2,  &  \text{if $m$ is even}.
 \end{cases}  $$

  It remains to prove that there are no other elements in
  $\Sigma^{tb}(m)$ except the skeleta of $\Delta^{[m]}$ described in the theorem. Let $K$ be an arbitrary element
  of $\Sigma^{tb}(m)$. \n
 
 When $m$ is odd, since both the maximal and the minimal elements of $\Sigma^{tb}(m)$ are $\Delta^{[m]}_{(\frac{m-3}{2})}$, $K$ can only be $\Delta^{[m]}_{(\frac{m-3}{2})}$.
  \n
  
  When $m$ is even, the maximal element of $\Sigma^{tb}(m)$ is $\Delta^{[m]}_{(\frac{m}{2}-1)}$ while the minimal element is $\Delta^{[m]}_{(\frac{m}{2}-2)}$. So $\dim(K)$ 
  can only be $\frac{m}{2}-1$ or $\frac{m}{2}-2$. If $\dim(K)=\frac{m}{2}-2$, then $\Delta^{[m]}_{(\frac{m}{2}-2)}\subseteq K$ implies $K=\Delta^{[m]}_{(\frac{m}{2}-2)}$. 
 But if $\dim(K)=\frac{m}{2}-1$, the situation is a bit complicated. \n
 
 In the following, suppose $K$ is a minimal element of $\Sigma^{tb}(m,\frac{m}{2}-1)$. 
     If $K$ has only one
 $(\frac{m}{2}-1)$-simplex, then clearly $\widetilde{tb}(K)$ is less than $\widetilde{tb}\big(\Delta^{[m]}_{(\frac{m}{2}-2)} \big)$, which contradicts our assumption that $K \in \Sigma^{tb}(m)$. So $K$ must have at least two
 $(\frac{m}{2}-1)$-simplices. We have the following two cases.\n
 
  \begin{itemize}
  \item[Case 1:] If all the $(\frac{m}{2}-1)$-simplices of $K$ are disjoint, then $K$ could only have two $(\frac{m}{2}-1)$-simplices. But it is easy to check that $\widetilde{tb}(K)$ is less than $\widetilde{tb}\big(\Delta^{[m]}_{(\frac{m}{2}-2)} \big)$, which again contradicts the assumption $K \in \Sigma^{tb}(m)$.\n
  
  \item[Case 2:]  Suppose $\sigma$ and $\tau$ are two
   $(\frac{m}{2}-1)$-simplices in $K$ which share exactly $s$ vertices where $1\leqslant s \leqslant \frac{m}{2}-1$. Without loss of generality, we can assume
    \nn
  \qquad\ \ $\sigma=\{1,\cdots, \frac{m}{2}\},  \
  \tau=\{\frac{m}{2}-s+1,\cdots, \frac{m}{2}, \cdots, m-s\}$.\n

    Let $\overline{\phi}$ be an arbitrary permutation
    of $[m]$ that preserves $\sigma$. Then $K\cap \overline{\phi}K$ contains $\sigma$, and hence $\dim(K\cap  \overline{\phi}K)=\frac{m}{2}-1$. But by the
    minimality of $K$ in
    $\Sigma^{tb}(m,\frac{m}{2}-1)$, we must have $ \overline{\phi}K = K$. In particular, $\overline{\phi}(\tau)$ belongs to $K$. This implies that any
    $(\frac{m}{2}-1)$-simplex in $\Delta^{[m]}$ that 
    shares exactly $s$ vertices with $\sigma$ 
    must also belong to $K$. This is because
    the permutation $\overline{\phi}$ can send the face
     $\sigma\cap \tau$
    to any other face of $\sigma$ with $s$ vertices
      while mapping
    the face $\tau\backslash \sigma$ of $\tau$
    to any other simplex with $\frac{m}{2}-s$ vertices
    in the complement of $\sigma$ in $\Delta^{[m]}$.
    So $\overline{\phi}(\tau)$ can exhaust all the $(\frac{m}{2}-1)$-simplices in $\Delta^{[m]}$ that
     share exactly $s$ vertices with $\sigma$.
    Moreover, we can similarly prove
    that any
    $(\frac{m}{2}-1)$-simplex in $\Delta^{[m]}$ that 
    shares exactly $s$ vertices with $\tau$ 
    must also belong to $K$. \n
    
  Next, we prove that there exists a $(\frac{m}{2}-1)$-simplex $\tau'$ in $K$ which
  shares a $(\frac{m}{2}-2)$-face with $\sigma$.
  Indeed, we can just trade the $\frac{m}{2}-s$ vertices  $\tau\backslash \sigma=\{\frac{m}{2}+1,\cdots, m-s \}$ with $\{2,\cdots, \frac{m}{2}-s\}\cup\{ m \}$ and obtain another $(\frac{m}{2}-1)$-simplex $\tau'=\{ 2,\cdots, \frac{m}{2}, m \}$ which satisfies 
  $|\tau'\cap \tau|=s$. So $\tau'$ must belong to $K$
  and we have $|\tau'\cap \sigma|= \frac{m}{2}-1$.
\end{itemize}
\n

 From the above discussion, we see that if $m$ is even and $K$ is a minimal element
  in $\Sigma^{tb}(m,\frac{m}{2}-1)$, 
 we can always find a pair of $(\frac{m}{2}-1)$-simplices $\sigma$ and $\tau$ in $K$ which
  share a $(\frac{m}{2}-2)$-face. 
   Moreover, by the argument in Case 2, any
  $(\frac{m}{2}-1)$-simplex of $\Delta^{[m]}$ that
  intersects $\sigma$ or $\tau$ at a $(\frac{m}{2}-2)$-face must belong to $K$. So without loss of generality, assume
  \n
  \qquad\qquad\qquad\quad\ \ \  $\sigma=\{ 1,\cdots, \frac{m}{2} \},\ \tau=\{ 2,\cdots, \frac{m}{2}+1\}$.
  \n
  
  \textbf{Claim:} Every $(\frac{m}{2}-1)$-simplex of $\Delta^{[m]}$ must belong to $K$, i.e. $K=\Delta^{[m]}_{(\frac{m}{2}-1)}$.\n
  
  If a $(\frac{m}{2}-1)$-simplex $\xi$ of $\Delta^{[m]}$
  is disjoint from $\sigma$, then $\xi=\{ \frac{m}{2}+1,\cdots, m\}$. We can construct a sequence of 
  $(\frac{m}{2}-1)$-simplices $\sigma=\xi_0$, $\xi_1,\cdots,\xi_{\frac{m}{2}} = \xi$ as follows:\n
  
 \qquad\qquad\qquad\qquad $\xi_i =\{ i+1, \cdots, \frac{m}{2}+ i\} $,  $1\leq i \leq \frac{m}{2}$.
  \n
  
  Since $\xi_i\cap \xi_{i+1}$ is a 
   $(\frac{m}{2}-2)$-simplex for every $1\leqslant i <\frac{m}{2}$, we can prove that $\xi_1,\cdots, \xi_n=\xi$ must all belong to $K$ by iteratively using the argument
   in Case 2.
  \n 
  
  Similarly, if $\xi\cap \sigma \neq \varnothing$, we can trade the vertices in $\xi\backslash\sigma$ with the vertices in 
 $\sigma\backslash \xi $ one at a time to turn
 $\sigma$ to $\xi$. This proves the claim. 
 \n
 
  By the above claim, the minimal element $K$ of $\Sigma^{tb}(m,\frac{m}{2}-1)$ is $\Delta^{[m]}_{(\frac{m}{2}-1)}$ which is also maximal. So
  $\Sigma^{tb}(m,\frac{m}{2}-1)$ consists only of $\Delta^{[m]}_{(\frac{m}{2}-1)}$.
 Therefore, when $m$ is even, 
  $\Sigma^{tb}(m)$ must either be
   $\Delta^{[m]}_{(\frac{m}{2}-2)}$
  or $\Delta^{[m]}_{(\frac{m}{2}-1)}$. 
  The theorem is proved. 
  \end{proof}
    
   For any finite simplicial complex $K$ of dimension $d$ and a field $\mathbb{F}$, one can associate two integer vectors $f=(f_0,f_1,\cdots,f_d)$ and $\beta=(\widetilde{\beta}_0,\widetilde{\beta}_1,\cdots,\widetilde{\beta}_d)$, where $f_i$ is the number of $i$-simplices of $K$ and $\widetilde{\beta}_i$ is the $i$-th reduced Betti number (over $\mathbb{F}$) of $K$ ($i=0,1,\cdots,d)$. A remarkable result proved in~\cite{BjrnKal88} tells us that a collection of nonlinear relations along with the linear relation given by the Euler-Poincar\'e formula completely characterize the integer vectors which can arise as $f$ and $\beta$ for a simplicial complex. But from the description in~\cite{BjrnKal88} of the
   range of the $f$-vectors and $\beta$-vectors of all the simplicial complexes with $m$ vertices,
   it is still not clear what the answer of
   Question 2 should be. Indeed, when we fix the dimension $d$ of the simplicial complex $K$, the upper bound of $\widetilde{tb}(K)$ will depend on $d$ and become a little  complicated (see Corollary~\ref{cor:num}). So to obtain the answer of Question 2,
    we also need to use the machinery of shifting of
    simplicial complexes and the theory of Sperner family just as~\cite{BjrnKal88} did. 
   \n
   
     Now, let us recall some basic definitions
     that are needed for our argument.

   \begin{defi}[Sperner Family]
   Let $X$ be a finite set.
	A \emph{Sperner family of $X$} is a set $\mathcal{F}$ of subsets of $X$ that satisfies $A\nsubseteq B$ for distinct members of $\mathcal{F}$. 	
	Given a subset $Y \subseteq X$, a \emph{Sperner family of $X$ over $Y$} is a Sperner family $\mathcal{F}$ of $X$ where every member of $\mathcal{F}$ has nonempty intersection with $Y$. 
\end{defi}

 The following is a fundamental result of Sperner on the
 size of a Sperner family (see Sperner~\cite{Spern28} or
 Anderson~\cite[Theorem 1.2.2]{And02}).
 
\begin{thm}[Sperner's Theorem]
Let $\mathcal{F}$ be a Sperner family of subsets of a finite set $X$ where
$|X|=n$. Then $|\mathcal{F}|\leq \binom{n}{[n/2]}$.
If $n$ is even, the only Sperner family consisting of $\binom{n}{[n / 2]}$ subsets of $X$ is made up of all the $\frac{n}{2}$-subsets of $X$. If $n$ is odd, a Sperner family of size $\binom{n}{[n / 2]}$ consists of either all the $\frac{1}{2}(n-1)$-subsets or all the $\frac{1}{2}(n+1)$-subsets of $X$.
\end{thm}

 A \emph{$k$-subset} of $X$ means a subset of order $k$. Another very useful construction in the study of combinatorics of simplicial complexes is the shifting operation.
 A \emph{shifting} operation is a map which assigns to every simplicial complex $K$ a shifted
simplicial complex $\Delta(K)$ with the same $f$-vector.

\begin{defi}[Shifted Complex]
 A simplicial complex $\Gamma$ with vertex set $[m]$ is
	called \emph{shifted} if for every simplex $\sigma=\{i_1,\cdots, i_s\}\in \Gamma$ where $i_1<\cdots<i_s$, any $\{ i'_1,\cdots, i'_s \}$ with $i'_1\leq i_1,\cdots, i'_s\leq i_s$ and
	$i'_1<\cdots<i'_s$ is also a simplex of $\Gamma$.	 
\end{defi}

A well-known (combinatorial) shifting operation, introduced by Erd\"os, Ko and Rado~\cite{ErKoRa61}, has been of great use in extremal set theory. Later, another shifting operation
 was introduced by Kalai in~\cite{Kal84} which preserves
 both the $f$-vector and the $\beta$-vector of a simplicial complex. 

\begin{thm}[see~\cite{Kal84} and~\cite{BjrnKal88}] \label{Thm:Shifted-Complex}
   Given a simplicial complex $K$ on $m$ vertices and a field $\mathbb{F}$, there exists a canonically defined shifted simplicial complex $\Delta=\Delta(K, \mathbb{F})$ on $[m]$ such that $f_i(\Delta;\mathbb{F})=f_i(K;\mathbb{F})$
   and $\widetilde{\beta}_i(\Delta;\mathbb{F})=\widetilde{\beta}_i(K;\mathbb{F})$ for all $i \geqslant 0$. 
\end{thm}

 Shifted complexes belong to a slightly larger class of
simplicial complexes called near-cones (see~\cite{BjrnKal88}).

\begin{defi}[Near-Cone]
	A simplicial complex $\Delta$ with vertex set $[m]$ is
	called a \emph{near-cone} if for every simplex $S \in \Delta$ and $2\leq j\leq m$, if $1 \notin S$ and $j \in S$, then $(S \backslash j) \cup \{1\} \in \Delta$. For a near-cone $\Delta$, define
\begin{equation} \label{Equ:B-Delta-def}
B(\Delta)=\{S \in \Delta \mid S \cup \{1\} \notin \Delta \} .
\end{equation}
\end{defi}

A very nice property of a near-cone $\Delta$ is that its total Betti number
can be easily computed by counting the number of elements in $B(\Delta)$.

\begin{lem}[\text{\cite[Lemma 4.2, Theorem 4.3]{BjrnKal88}}]
\label{Lem:B-Delta}
	If $\Delta$ is a near-cone on $[m]$, then
	\begin{itemize}
		\item[(i)] $B(\Delta)$ is a Sperner family of $\{ 2,\cdots,m \}$,
		\item[(ii)] every simplex $S \in B(\Delta)$ is maximal in $\Delta$,
		\item[(iii)] $\wt{tb}(\Delta)=\vert B(\Delta)\vert$.
	\end{itemize}
\end{lem}

The following proposition tells one what kind of 
Sperner families on $[m]$ can be realized by $B(\Delta)$ 
of a near-cone $\Delta$. 

\begin{prop}\label{prop:equiv}
	Let $\mathcal{F}$ be a Sperner family of $[m]\backslash \{1\}=\{2,\cdots,m\}$. Then the following statements are equivalent:
	\begin{itemize}
		\item[(i)] there exists a $d$-dimensional near-cone $\Delta$ with vertex set contained in $[m]$, such that $B(\Delta)=\mathcal{F};$
		\item[(ii)] there exists a subset $\{i_1,\cdots,i_d\}\subseteq \{2,\cdots,m\}$ such that no subset of $\{i_1,\cdots,i_d\}$ belongs to $\mathcal{F}$ and the order of each member of $\mathcal{F}$ is no greater than $d+1$.
	\end{itemize}
\end{prop}
\begin{proof}
	(i)$\Rightarrow$(ii). First of all, since $\Delta$ is a $d$-dimensional near-cone, it has at least one $d$-simplex, say $\{1,i_1,\cdots,i_d\}$. Then by the definition of $B(\Delta)$ in~\eqref{Equ:B-Delta-def}, no subset of $\{i_1,\cdots,i_d\}$ belongs to $\FF=B(\Delta)$ since $\{1,i_1,\cdots,i_d\}$ is a face of $\Delta$. Moreover, the order of each member of $\mathcal{F}$ is no greater than $d+1$ since $\Delta$ is $d$-dimensional.
	\n
	(ii)$\Rightarrow$(i). From the Sperner family $\FF$, we define 
	\begin{equation} \label{Equ:Delta-F}
	  \Delta=\left[1*\left(E(\mathcal{F})\setminus\mathcal{F}\right)\right]\cup\mathcal{F}\cup\Delta^{\{1,i_1,\cdots,i_d\}},
	  \end{equation}
	where $E(\mathcal{F})$ is the simplicial complex generated by $\mathcal{F}$, that is, the minimal simplicial complex taking all the members of $\mathcal{F}$ as its maximal faces. 
	Then $\Delta$ is clearly a $d$-dimensional simplicial complex whose vertex set is a subset of $[m]$. Moreover, for any $ F\in\mathcal{F}$ and $j\in F$, we have $(F\setminus j)\cup 1\in 1*\left(E(\mathcal{F})\setminus\mathcal{F}\right)$,
	so $\Delta$ is a near-cone. Finally, it is obvious that $B(\Delta)=\mathcal{F}$.
\end{proof}

\begin{defi}
 Let $X$ be a finite set with order $|X|=n$. For a nonempty subset $Y$ of $X$, let $C(n,X,Y)$ denote the set of all subsets of $X$ that have nonempty intersection with $Y$. For any $i\geq 1$, let $C_i(n,X,Y)$ denote the collection of sets in $C(n,X,Y)$ of size $i$.
 \end{defi}

The following lemma combines some results from Lih~\cite{Lih80} and Griggs~\cite{Gri82}, where 
the $\lceil \cdot \rceil$ is the ceiling function.

\begin{lem}[see~\text{\cite[Theorem 2]{Lih80} and \cite[Theorem 7.1]{Gri82}}]\label{lem:sper}
 Let $X$ be a finite set of order $n$.
	The maximal possible cardinality of a Sperner family $\FF$ of $X$ over a subset $Y\subseteq X$ with $\vert Y\vert=k$ is $f(n, k)=\binom{n}{\lceil n / 2\rceil}-\binom{n-k}{\lceil n/2\rceil}$. 
	Moreover, $\vert\FF\vert=f(n,k)$ if and only if $\FF$ is one of the following cases:
	\begin{enumerate}
		\item[(a)] $C_{\left\lceil\frac{1}{2} n\right\rceil}(n,X,Y)$;
		\item[(b)] $C_{\frac{1}{2}(n-1)}(n,X,Y)$, for odd $n$ and  $k \geqslant \frac{1}{2}(n+3)$;
		\item[(c)] $C_{\frac{1}{2}(n+2)}(n,X,Y)$, for even $n$ and $k=1$.
	\end{enumerate}
	In particular, if $\vert\FF\vert=f(n,k)$, then every member in $\FF$ has the same order.
\end{lem}

 Now, we are ready to prove the following theorem which answers Question 2.
 
 \begin{thm}\label{Thm-Main-2}
The sets $\Sigma^{tb}(m,d)$ are classified as follows:
 \begin{itemize}
  \item[(i)] If $d \leqslant\left[\frac{m}{2}\right]-1$ or $d=m-1$, then $\Sigma^{tb}(m,d)=\Big\{\Delta_{(d)}^{[m]}\Big\}$;
  \item[(ii)] If $\left[\frac{m}{2}\right] \leqslant d \leqslant m-3 $, then $\Sigma^{tb}(m,d)=
  \Big\{\Delta_{\left(\left[\frac{m}{2}\right]-1\right)}^{[m]}\langle d\rangle\Big\}$;
  \item[(iii)] If $d=m-2$,
    \begin{itemize}
      \item when $m$ is odd, $\Sigma^{tb}(m,d)=\Big\{\Delta_{\left(\left[\frac{m}{2}\right]-1\right)}^{[m]}\langle d\rangle, \Delta_{\left(\left[\frac{m}{2}\right]\right)}^{[m]}\langle d\rangle\Big\}$;
      
 \item when $m$ is even, $\Sigma^{tb}(m,d)=\Big\{\Delta_{\left(\left[\frac{m}{2}\right]-1\right)}^{[m]}\langle d\rangle \Big\}$.
  \end{itemize}
  \end{itemize}
\end{thm}
 
 \begin{proof}
 By the algebraic shifting construction in \cite[Theorem 3.1]{BjrnKal88}, there is a unique shifted complex $\Delta$ associated to $K\in\Sigma(m,d)$ where $\Delta$ has the same $f$-vector and $\beta$-vector as $K$. Moreover, $\Delta$ is a near-cone. So by Lemma~\ref{Lem:B-Delta},
  $$\widetilde{tb}(K) = \widetilde{tb}(\Delta) =
    |B(\Delta)|,$$
    where $B(\Delta)$ is a Sperner family of $\{2,\cdots,m\}$. Moreover, since $\Delta$ is shifted and
    $\dim(\Delta)=\dim(K)=d$, we can assume that $\{1,\cdots,d+1\}$ is a $d$-simplex of $\Delta$.	\n  
	
 (i)\, The case $d=m-1$ is trivial. We do induction
 on $d\leq\left[\frac{m}{2}\right]-1$. The case $d=0$
 is clearly true.  Let $K$ be an arbitrary element of $\Sigma^{tb}(m,d)$. For any permutation $\phi$ of $[m]$, since
 $K\cup \phi K \in \Sigma(m,d)$, we have
	$$\widetilde{tb}(K)=\widetilde{tb}(\phi K) \geq
	\widetilde{tb}(K\cup \phi K).$$
	So we can deduce from~\eqref{Equ:tb-K} that
	$\widetilde{tb}(K)\leq \widetilde{tb}(K\cap\phi K)$.
	\nn
	
	\textbf{Claim:} $\dim(K\cap\phi K)=d$, hence $K\cap\phi K\in \Sigma^{tb}(m,d)$.\nn
	
 Assume that $s=\dim(K\cap \phi K)<d$.  Then since $K\in\Sigma^{tb}(m,d)$, we have $\widetilde{tb}(K) \geq \widetilde{tb}(\Delta_{(d)}^{[m]})=\binom{m-1}{d+1}$, and hence
	$ \widetilde{tb}(K\cap\phi K) \geq \binom{m-1}{d+1}$.
 On the other hand, by our induction hypothesis 
	$\widetilde{tb}(K\cap\phi K) \leq \widetilde{tb}(\Delta_{(s)}^{[m]}) = \binom{m-1}{s+1}
	$. 
	\begin{itemize}
	\item When $m$ is odd, since $s< d\leq \left[\frac{m}{2}\right]-1$,	$\binom{m-1}{s+1} < \binom{m-1}{d+1}$, a contradiction. So we must have $s=d$. \n
 
  \item When $m$ is even and $d \leq \left[\frac{m}{2}\right]-2$, the same argument as the previous case applies.
 So the only remaining case is $d=\left[\frac{m}{2}\right]-1=\frac{m}{2}-1$. Since $B(\Delta)$ is a Sperner family of $\{2,\cdots,m\}$, by Sperner's theorem we have
   $$ \widetilde{tb}(K) = \widetilde{tb}(\Delta) =
    |B(\Delta)| \leq \binom{m-1}{\left[\frac{m-1}{2}\right]}
    = \binom{m-1}{\frac{m}{2}-1} $$ 
 where the equality holds if and only if $B(\Delta)$
 consists of either all the $(\frac{m}{2}-1)$-subsets or all the $\frac{m}{2}$-subsets of $\{ 2,\cdots, m \}$.
 But since the dimension $\dim(\Delta)=\dim(K)=\frac{m}{2}-1$,
 $B(\Delta)$ must be the later case. This implies that $K = \Delta_{(\frac{m}{2}-1)}^{[m]}$ and so $K\cap \phi K = K$.
 The claim is proved.
  \end{itemize}

		 By the above claim, if $K$ is a minimal element of
		 $\Sigma^{tb}(m,d)$, then $K\cap\phi K=K$ for every permutation $\phi$ of $[m]$, which implies that $K=\Delta_{(d)}^{[m]}$. But $\Delta_{(d)}^{[m]}$ is maximal in $\Sigma(m, d)$, so we can assert that $\Sigma^{tb}(m,d)=\Big\{\Delta_{(d)}^{[m]}\Big\}$.\n
		 
		 (ii) and (iii)\, When $ \left[\frac{m}{2}\right] \leq d \leq m-2$, let 
		  $$X=\{2,\cdots, m\}, \ \ Y=[m]\setminus\{1,\cdots d+1\} =\{d+2,\cdots,m\}.$$
		 	 
	Observe that $B(\Delta)$ is a Sperner family of $X$ that satisfies the condition (2) in Proposition~\ref{prop:equiv} with $\{i_1,\cdots, i_d\}=
		  \{2,\cdots,d+1\}$. Therefore, no subset of $\{2,\cdots,d+1\}$ belongs to $B(\Delta)$, which implies that
		 every member of $B(\Delta)$ has nonempty intersection 
		 with $Y$. In other words, $B(\Delta)$ is a Sperner family of $X$ over $Y$. So by Lemma~\ref{lem:sper},
		 the cardinality of $B(\Delta)$ satisfies:
		 $$\vert B(\Delta) \vert \leq \binom{m-1}{\lceil \frac{m-1}{2}\rceil}-\binom{d}{\lceil \frac{m-1}{2}\rceil}=\binom{m-1}{\left[\frac{m}{2}\right]}-\binom{d}{\left[ \frac{m}{2}\right]} $$
 where the equality holds if and only if $B(\Delta)$ is one of the three types of Sperner families listed in Lemma~\ref{lem:sper} with $n=|X|=m-1$ and $k=|Y|=m-d-1$.
 Then since $|B(\Delta)|$ computes $\widetilde{tb}(K)$, the
  Sperner families described in Lemma~\ref{lem:sper} will give us all possible members of
   $\Sigma^{tb}(m,d)$.
 \begin{itemize}
 \item In the case (a) of Lemma~\ref{lem:sper}, the near-cone constructed from the Sperner family $C_{\left\lceil\frac{1}{2} n\right\rceil}(n,X,Y)$ in~\eqref{Equ:Delta-F} in Proposition~\ref{prop:equiv} is exactly $\Delta^{[m]}_{\left(\left[\frac{m}{2}\right]-1\right)}\langle d\rangle$ for every $\left[\frac{m}{2}\right]\leq d \leq m-2$. \n
 
 \item In the case (b) of Lemma~\ref{lem:sper}, the near-cone constructed from the Sperner family $C_{\frac{1}{2}(n-1)}(n,X,Y)$ in~\eqref{Equ:Delta-F} has dimension $n-k\leq n-\frac{1}{2}(n+3)=\frac{m}{2}-2$ which contradicts our assumption $d\geq \left[\frac{m}{2}\right]$, hence is invalid. \n
 
 \item In the case (c) of Lemma~\ref{lem:sper}, $n$ is even and then $m$ is odd. 
 The near-cone constructed from the Sperner family $C_{\frac{1}{2}(n+2)}(n,X,Y)$ in~\eqref{Equ:Delta-F} has dimension $d=n-k=m-2$, which gives rise to $\Delta^{[m]}_{\left(\left[\frac{m}{2}\right]\right)}\langle m-2\rangle$. 
 \end{itemize}
 
 From the above discussion, we obtain the desired 
 statements in (ii) and (iii). 
\end{proof}

The following is an immediate corollary of the proof of Theorem~\ref{Thm-Main-2}.

\begin{cor}\label{cor:num}
	For any $d$-dimensional simplicial complex $K$ with  
	$m$ vertices, the upper bound of $\widetilde{tb}(K)$ is given by:
	\begin{enumerate}
		\item[(i)] $\widetilde{tb}(K)\leq \binom{m-1}{d+1},$ if $d\leq\left[\frac{m}{2}\right]-1$;
		\item[(ii)] $\widetilde{tb}(K)\leq \binom{m-1}{[m/2]}-\binom{d}{[m/2]},$ if $\left[\frac{m}{2}\right] \leq d \leq m-2$;
		\item[(iii)] $\widetilde{tb}(K)=0$, if $d=m-1$.
	\end{enumerate}
	Moreover, the equalities in $\mathrm{(i)}$ and $\mathrm{(ii)}$ hold if and only if $K$ belongs to $\Sigma^{tb}(m,d)$ as described in Theorem~\ref{Thm-Main-2}. 
\end{cor}
\vskip .4cm

  \section{Simplicial complexes with the minimum total bigraded Betti number} \label{Sec-Minimum-D}
  
  For any $d$-dimensional simplicial complex $K$ with $m$ vertices, there is a universal lower bound of $\widetilde{D}(K)$ that depends only on $m$ and $d$.
  This lower bound was discovered through an interesting relation between $\widetilde{D}(K)$ and some canonical CW-complex associated to $K$ called \emph{real moment-angle complex} of $K$ (see Davis and Januszkiewicz~\cite[p.\,428--429]{DaJan91} or Buchstaber and Panov~\cite[Section 4.1]{BP15}). One way to write $\R\mathcal{Z}_K$ is
  \begin{equation} \label{Equ:RZK} 
  \R\mathcal{Z}_K =  \bigcup_{\sigma\in K} \Big( \prod_{i\in\sigma} D^1_{(i)} \times \prod_{i\notin \sigma} S^0_{(i)} \Big) \subseteq \prod_{i\in [m]} D^1_{(i)} =[-1,1]^m,
  \end{equation}
  where $D^1_{(i)}$ and $S^0_{(i)}$ is a copy of $D^1=[-1,1]$ and $S^0=\partial D^1 = \{-1,1\}$ indexed by $i\in [m]$, and $\prod$ denotes Cartesian product of spaces.\n

 It is shown in~\cite[Section 4]{BP15} (also see~\cite[Theorem 4.2]{CaoLu12}) that
 $\mathrm{Tor}_{\mathbb{F}[v_1,\cdots, v_m]}(\mathbb{F}[K],\mathbb{F})$ computes the cohomology groups of $\R\mathcal{Z}_K$, which implies 
  \begin{equation} \label{Equ:DK-ZK}
    \widetilde{D}(K;\mathbb{F}) =\dim_{\mathbb{F}} \mathrm{Tor}_{\mathbb{F}[v_1,\cdots, v_m]}(\mathbb{F}[K],\mathbb{F}) =tb(\R\mathcal{Z}_K;\mathbb{F}). 
   \end{equation}

 Moreover, by~\cite[Theorem 1.4]{CaoLu12} or~\cite[Theorem 3.2]{Uto11},
 there is a universal lower bound of $ tb(\R\mathcal{Z}_K;\mathbb{F})$ for any simplicial complex $K$:
  \begin{equation} \label{Equ:Lower-Bound}
   tb(\R\mathcal{Z}_K;\mathbb{F}) \geqslant 2^{m-\dim(K)-1}.
   \end{equation}  
  So for any $d$-dimensional simplicial complex $K$ with
  $m$ vertices, we always have
  \begin{equation} \label{Equ:Uni-Lower-bound}
  \widetilde{D}(K;\mathbb{F}) \geqslant 2^{m-d-1}.
  \end{equation}
  
 In the following, we study those simplicial complexes
  that make the equality in~\eqref{Equ:Lower-Bound} hold, i.e. tight simplicial complexes (see Definition~\ref{Def:Tight-Complex}). The coefficient $\mathbb{F}$ will be omitted in the rest of
   this section.\nn
   
    The following are some easy lemmas on the properties of $\widetilde{D}(K)$.
    
 \begin{lem} \label{Lem-retr}
	If $K'$ is a full subcomplex of $K$, then
	$\widetilde{D}(K')\leq \widetilde{D}(K)$.
	\end{lem}
\begin{proof}
 This follows from the formula~\eqref{Equ:D-tilde-K-J} of $\widetilde{D}(K)$ and the simple fact that a full subcomplex of $K'$ is also a full subcomplex of $K$. 
 \end{proof}

\begin{lem} \label{Lem:Simplex}
  Let $K$ be a simplicial complex with vertex set $[m]=\{1,2,\cdots, m\}$. Then $\widetilde{D}(K)\geqslant 1$. Moreover,
  $\widetilde{D}(K)=1$ if and only if $K$ is the simplex $\Delta^{[m]}$. 
 \end{lem}
 \begin{proof}
 We use the formula of $\widetilde{D}(K)$ in ~\eqref{Equ:D-tilde-K-J}.
 Consider a minimal subset $J\subseteq [m]$ that does not span a simplex in $K$ (called a minimal non-face). If $J$ is not the empty set, then $K|_J=\partial\Delta^J$ and
  $\widetilde{tb}(K|_J) = 1$. This implies 
  $$\widetilde{D}(K)\geq  \widetilde{tb}(K|_{\varnothing}) + \widetilde{tb}(K|_J) =  2.$$
  So $\widetilde{D}(K)=1$ if and only if all the minimal non-faces of $K$ are empty, i.e. $K$ is the simplex
  $\Delta^{[m]}$.
 \end{proof}

     \begin{lem} \label{Lem:Join-X}
  For any finite CW-complexes $X$ and $Y$,  
    $$  tb(X\times Y) = tb(X)tb(Y), \ \ \ \widetilde{tb}(X*Y) = \widetilde{tb}(X) \widetilde{tb}(Y)$$
    where $X*Y$ is the join of $X$ and $Y$.
  \end{lem}
  \begin{proof}
   The equality $tb(X\times Y) = tb(X)tb(Y)$ follows from the K\"unneth formula of homology groups.
    In addition, by the homotopy equivalence 
    $$X*Y\simeq \Sigma(X\wedge Y)$$
     where ``$\wedge$'' is the smash product and ``$\Sigma$'' is the (reduced) suspension, we obtain (with a field coefficient) that
    \begin{equation} \label{Equ:Hom-Iso-join}
      \widetilde{H}_n(X*Y) \cong H_{n-1}(X\wedge Y) \cong
    \bigoplus_{i} (\widetilde{H}_i(X) \otimes \widetilde{H}_{n-1-i}(Y) ). 
    \end{equation}
    The second isomorphism in~\eqref{Equ:Hom-Iso-join}
    follows from the relative version of the K\"unneth formula (see~\cite[Corollary 3.B.7]{Hat02}).
   Notice that $X*Y$ is always path-connected and hence $\widetilde{H}_0(X*Y)=0$. Then it follows that $\widetilde{tb}(X*Y) = \widetilde{tb}(X) \widetilde{tb}(Y)$.
  \end{proof}

 \begin{lem} \label{Lem:Join-K}
  For any finite nonempty simplicial complexes $K$ and $L$,  
  $$ \widetilde{D}(K*L) = \widetilde{D}(K) \widetilde{D}(L) .  $$
  \end{lem}
 \begin{proof}
   Since $\R\mathcal{Z}_{K*L} \cong \R\mathcal{Z}_K\times \R\mathcal{Z}_L$ (see~\cite[Chapter\,4]{BP15}), we obtain from~\eqref{Equ:DK-ZK} and
   Lemma~\ref{Lem:Join-X} that
   \[ \widetilde{D}(K*L) = tb(\R\mathcal{Z}_{K*L}) =tb(\R\mathcal{Z}_K\times \R\mathcal{Z}_L)  
   =tb(\R\mathcal{Z}_K) tb(\R\mathcal{Z}_L) = \widetilde{D}(K) \widetilde{D}(L).  \]
 \end{proof}
 
  Observe that if $K$ is a $d$-dimensional simplicial complex with $m$ vertices, then for any $r\geqslant 1$,
  $\Delta^{[r]}*K$ is a
  $(d+r)$-dimensional simplicial complex with $m+r$ vertices and, $\widetilde{D}(\Delta^{[r]}*K)=\widetilde{D}(K)$ by Lemma~\ref{Lem:Join-K}. 
   So we obtain the corollary immediately.
 
 \begin{cor} \label{Cor:Tight-Join}
 A finite simplicial complex $K$ is tight if and only if
 $\Delta^{[r]}*K$ is tight for all $r\geq 1$.
 \end{cor}

  By Yu and Masuda~\cite[Proposition 2.1]{YuMas21}, any
  simplicial complex of the form $\partial \Delta^{[n_1]}*\cdots*\partial \Delta^{[n_k]}$
  is tight. So by Corollary~\ref{Cor:Tight-Join},
  $\Delta^{[r]}*\partial \Delta^{[n_1]}*\cdots*\partial \Delta^{[n_k]}
  $ is also tight for all $r\geq 1$.
   Then one may ask whether simplicial complexes of the form
    $\partial \Delta^{[n_1]}*\cdots*\partial \Delta^{[n_k]}$ or $\Delta^{[r]} * \partial \Delta^{[n_1]}*\cdots*\partial \Delta^{[n_k]}
  $ are all the tight simplicial complexes?
  We will see in Theorem~\ref{Thm-Main-3} that the answer is yes.\n
  
   For brevity, we introduce the following terms.
   
   \begin{defi}
    For any positive integers $n_1,\cdots, n_k$ and $r$, we call the simplicial complex $\partial \Delta^{[n_1]}*\cdots*\partial \Delta^{[n_k]}$ a \emph{sphere join}
   and call $\Delta^{[r]}*\partial \Delta^{[n_1]}*\cdots*\partial \Delta^{[n_k]}
  $ a \emph{simplex-sphere join}.
   \end{defi}
   
  The following theorem proved
 in~\cite{YuMas21} will be useful
 for our proof. 
   
  \begin{thm}[\text{\cite[Theorem 3.1]{YuMas21}}]  \label{Thm:Yu-Masuda}
   Let $K$ be a simplicial complex of dimension $n\geq 2$. Suppose that $K$ satisfies the following two conditions:
 \begin{itemize}
  \item[(a)] $K$ is an $n$-dimensional pseudomanifold,
 \item[(b)] the link of any vertex of $K$ is a sphere join of dimension $n-1$,
 \end{itemize}
Then $K$ is a sphere join.
\end{thm}

Recall that $K$ is an $n$-dimensional \emph{pseudomanifold} if
the following conditions hold:
\begin{itemize}
\item[(i)] Every simplex of $K$ is
a face of some $n$-simplex of $K$ (i.e. $K$ is \emph{pure}).\n

\item[(ii)] Every  $(n-1)$-simplex of $K$ is the face of exactly two $n$-simplices of $K$.\n

\item[(iii)] If $\sigma$ and $\sigma^{\prime}$ are two 
$n$-simplices of $K$, then there is a finite sequence of $n$-simplices $\sigma=\sigma_0, \sigma_1, \ldots, \sigma_k=\sigma^{\prime}$  such that the intersection $\sigma_i \cap \sigma_{i+1}$ is an $(n-1)$-simplex for all $i=0, \ldots, k-1$.
\end{itemize}
In particular, any closed connected PL-manifold is a pseudomanifold.\n

In addition, we will use the following inequality proved in~\cite[Theorem 3.2]{Uto11}:
\begin{equation} \label{Equ:uto-mdim}
  \widetilde{D}(K)= tb(\R\mathcal{Z}_K) \geq 2^{m-\mathrm{mdim}(K)-1},
\end{equation}
where $m$ is the number of vertices of $K$, and
$$ \mathrm{mdim}(K)=\ \text{the minimal dimension of the maximal simplices of}\ K.$$ 
The inequality~\eqref{Equ:uto-mdim} refines
the inequality~\eqref{Equ:Lower-Bound} since
$\mathrm{mdim}(K)\leq \dim(K)$.

\begin{lem}\label{lem:uto}
 	Let $K$ be a simplicial complex with $m$ vertices.  If $K$ is tight, then
 	\begin{itemize}
 		\item[(i)] $K$ is pure.
 		\item[(ii)] For every simplex $\sigma$ of $K$, $\mathrm{Link}_K \sigma$ is tight. 
 	\end{itemize}
 \end{lem}
\begin{proof}
 (i) By~\cite[Theorem 3.2]{Uto11},
   $\widetilde{D}(K) = tb(\R\mathcal{Z}_K) \geq 2^{m-\mathrm{mdim}(K)-1} \geq 2^{m-\mathrm{dim}(K)-1}$.
  Then since $K$ is tight, we must have $\mathrm{mdim}(K)=\mathrm{dim}(K)$, which implies that every maximal simplex of $K$ has the same dimension as $K$. So $K$ is pure.\n
 
  (ii) We do induction on $m$. When $m=1$, this is trivial. For any vertex $v$ of $K$, let $m_v$ be the number of vertices in $\mathrm{Link}_K v$. By the 
    proof of~\cite[Theorem 3.2]{Uto11} (note that 
     the argument there works for any vertex of $K$), there is a subspace $X_v$ of 
   $\R\mathcal{Z}_K$ with
   $tb(\R\mathcal{Z}_K)  \geqslant tb(X_v)$,
   where $X_v$ is homeomorphic to the disjoint union of $2^{m-m_v-1}$ copies of 
   $\R\mathcal{Z}_{\mathrm{Link}_K v}$.
   So we have
   $$ 2^{m-\mathrm{dim}(K)-1}=tb(\R\mathcal{Z}_K) \geqslant 2^{m-m_v-1}
   tb(\R\mathcal{Z}_{\mathrm{Link}_K v}). $$
   Then since $\dim(\mathrm{Link}_K v)\leq \dim(K)-1$, we obtain
      $$tb(\R\mathcal{Z}_{\mathrm{Link}_K v})\leq
      2^{m_v-\dim(K)} \leq 2^{m_v-\dim(\mathrm{Link}_K v)-1}.$$
  On the other hand,  by our induction we have
  $tb(\R\mathcal{Z}_{\mathrm{Link}_K v}) \geq 
  2^{m_v-\dim(\mathrm{Link}_K v)-1}$ since $m_v < m$.    
   So $\mathrm{Link}_K v$ is tight.\n
    
   Now suppose we have proved that $\mathrm{Link}_K \sigma$ is tight for any simplex $\sigma$ of $K$ with dimension less than $j$.
    Let $\tau$ be a $j$-simplex in $K$ and let $v$ be a vertex of $\tau$. So $\sigma = \tau \backslash \{v\}$ is a $(j-1)$-simplex and it is easy to see that
      \begin{equation} \label{Equ:link-link}
         \mathrm{Link}_K \tau = \mathrm{Link}_{\mathrm{Link}_K\sigma} v. 
         \end{equation}
   By our assumption, we know that $\mathrm{Link}_K \sigma$ is tight. So by our preceding argument, 
    we can assert that $\mathrm{Link}_K \tau$ is also tight. This finishes the proof.
  \end{proof}

\begin{lem} \label{Lem:S0}
  If a simplicial complex $K$ is tight but not connected, then $K$ must be $S^0=\partial \Delta^{[2]}$.
\end{lem}
\begin{proof}
  If $\dim(K)=0$, then $K$ being tight implies that $K$ is  isomorphic either to $\Delta^{[1]}$ or to $\partial \Delta^{[2]}$.
  If $K\in \Sigma(m,d)$ with $\dim(K)=d\geq 1$ and $K$ is not connected, let $K=K_1\sqcup K_2$ where $K_1$ and $K_2$ are two subcomplexes of $K$ that are disjoint. Then we can add a $1$-simplex $\sigma=\{i_1, i_2\}$ to $K$ with a vertex $\{i_1\}\in K_1$ and $\{i_2\}\in K_2$ and obtain a new simplicial complex $K'\in\Sigma(m,d)$.
 Observe that adding the $1$-simplex $\sigma$ to $K$ kills the generator of $\widetilde{H}_0(K|_{\{i_1,i_2\}})= \widetilde{H}_0(S^0)$. So for any $J\subset [m]$, it is easy to see that
 $$ \widetilde{tb}(K'|_J)= \begin{cases}
   \widetilde{tb}(K|_J)-1,  &  \text{if $J$ contains $\{i_1,i_2\}$}; \\
   \widetilde{tb}(K|_J),  &  \text{otherwise}.
 \end{cases}$$ 
 So by the formula~\eqref{Equ:D-tilde-K-J} of $\widetilde{D}(K)$, we can deduce that $\widetilde{D}(K') < \widetilde{D}(K)$. But since $K$ is tight, 
 $\widetilde{D}(K') < \widetilde{D}(K)=2^{m-d-1}$
  contradicting the inequality in~\eqref{Equ:Uni-Lower-bound}. 
\end{proof}

 Now, we are ready to prove the following theorem which answers Question 3.
  
  \begin{thm} \label{Thm-Main-3}
   A finite simplicial complex
     $K$ is tight if and only if $K$ is
     of the form $ \partial \Delta^{[n_1]}*\cdots*\partial \Delta^{[n_k]}$ or $\Delta^{[r]} * \partial \Delta^{[n_1]}*\cdots*\partial \Delta^{[n_k]}$ for some positive integers
     $n_1,\cdots, n_k$ and $r$.
  \end{thm}
  \begin{proof}
  By our preceding discussion, any sphere join or simplex-sphere join is tight. Conversely, suppose $K\in \Sigma(m,d)$ is tight and we do induction on $m$. When $m\leq 2$, the statement is trivial. For $m\geq 3$, by Lemma~\ref{lem:uto} and Lemma~\ref{Lem:S0},
 $K$ is connected, pure and $\mathrm{Link}_K \sigma$ is tight for every simplex $\sigma$ of $K$.  Then since the number of vertices of $\mathrm{Link}_K \sigma$ is less than $m$,
  our induction hypothesis implies that
  $\mathrm{Link}_K \sigma$ is either a sphere join or
  a simplex-sphere join. \n
  
  \textbf{Case 1:} If for every vertex $v$ of $K$, $\mathrm{Link}_K v$ is a sphere join, then by the relation in~\eqref{Equ:link-link}, we can inductively prove that $\mathrm{Link}_K \sigma$ is also a sphere join for every simplex $\sigma$ of $K$. This implies that $K$ is a closed connected PL-manifold hence a pseudomanifold. So by Theorem~\ref{Thm:Yu-Masuda}, 
  $K$ is sphere join.\n
  
  \textbf{Case 2:} If there exists a vertex $v\in K$ such that $\mathrm{Link}_K v$ is a simplex-sphere join, let
  $\mathrm{Link}_K v=\Delta^{[r]}*\partial\Delta^{[n_1]}*\cdots*\partial\Delta^{[n_k]}$.
 Since $K$ is a pure $d$-dimensional simplicial complex, $\dim(\mathrm{Link}_K v)=d-1$. 
	Take a vertex $w\in\Delta^{[r]}$ and consider the full subcomplex $K\backslash w:=K|_{[m]\setminus w}$ of $K$. Note that 
	$$\mathrm{Link}_{K\backslash w} v=\Delta^{[r]\setminus w} *\partial\Delta^{[n_1]}*\partial\Delta^{[n_2]}*\cdots*\partial\Delta^{[n_k]},$$
	which has dimension $d-2$. So the dimension of 
	$\mathrm{Star}_{K\backslash w} v$ is $d-1$, which implies 
	$$\mathrm{mdim}(K\backslash w)\leq d-1.$$
	But removing a vertex can reduce $\mathrm{mdim}(K)$ at most by one. Then since $K$ is pure, $\mathrm{mdim}(K\backslash w) \geq \mathrm{mdim}(K)-1= \dim(K)-1=d-1$.
	So $\mathrm{mdim}(K\backslash w)=d-1$. Then by~\eqref{Equ:uto-mdim}, we obtain
	$$
  \widetilde{D}(K\backslash w) \geq 2^{m-1-\mathrm{mdim}(K\setminus w)-1} = 2^{m-d-1}.$$
  But by Lemma~\ref{Lem-retr}, $\widetilde{D}(K\backslash w) \leq \widetilde{D}(K) = 2^{m-d-1}$. So we have
   \begin{equation} \label{Equ:D-equal}
      \widetilde{D}(K\backslash w)  = 2^{m-d-1} = \widetilde{D}(K). 
   \end{equation}   
   Moreover, by the formula~\eqref{Equ:D-tilde-K-J} of $\widetilde{D}(K)$, we obtain
      $$\widetilde{D}(K) - \widetilde{D}(K\backslash w)
  = \sum_{w\in J\subseteq [m]} \widetilde{tb}(K|_J).$$ 
  So the equality~\eqref{Equ:D-equal} implies that $\widetilde{tb}(K|_J)=0$  for every $J\subset [m]$ that contains $w$.\n
  
  \textbf{Claim:} For any simplex $\sigma$ of $K\backslash w$, $\{w\}\cup \sigma$ is a simplex of $K$.\n
  
   	We prove the claim by induction on $|\sigma|$. When $|\sigma|=1$, i.e. $\sigma$ is a vertex, $\{w\}\cup \sigma\in K$ since $\widetilde{tb}(K|_{\{w\}\cup \sigma})=0$. Assume that the claim is true when $|\sigma|< s$. If $|\sigma|=s$, then by induction $K|_{\{w\}\cup \tau}$ is a simplex for every $\tau\subsetneq \sigma$. If $\{w\}\cup \sigma$ is not a simplex of $K$,
   	 then $K|_{\{w\}\cup \sigma}$ is isomorphic to the boundary of an
   	 $s$-dimensional simplex. But this contradicts the above conclusion that 
   	 $\wt{tb}(K|_{\{w\}\cup \sigma})=0$. So the claim is proved.
  
    \n
    
    By the above claim, $K=w*(K\backslash w)$ is a cone
    of $w$ with $K\backslash w$. It follows that
    $$  \dim(K\backslash w) = d-1.$$ 
    So by~\eqref{Equ:D-equal}, $ \widetilde{D}(K\backslash w)  = 2^{m-d-1} = 2^{m-1-\dim(K\backslash w)-1}$, i.e. $K\backslash w$ is tight. Then by our induction, $K\backslash w$
    is either a sphere-join or a simplex-sphere join, which implies that $K=w*(K\backslash w)$ is a simplex-sphere join. The theorem is proved.
\end{proof}

   \vskip .4cm
   
     \section{Simplicial complexes with the maximum total bigraded Betti number} \label{Sec-Maximum-D}

 In this section, we give a complete answer to Question 4.
 First, we prove a lemma parallel to Lemma~\ref{lem:mv}
 on total bigraded Betti number.
 
\begin{lem}\label{lem:mv2}
	For any two simplicial complexes $K,L$ with vertex set $[m]$, 
	$$\widetilde{D}(K)+\widetilde{D}(L) \leq \widetilde{D}(K \cap L)+ \widetilde{D}(K \cup L).$$
\end{lem}
\begin{proof}
 From~\eqref{Equ:RZK}, it is easy to see that
	$$\R\ZZ_{K\cup L}=\R\ZZ_K\cup \R\ZZ_L, \quad \R\ZZ_{K\cap L}=\R\ZZ_K\cap \R\ZZ_L.$$
	Then similarly to Lemma~\ref{lem:mv}, the Mayer-Vietoris sequence for $(\R\ZZ_K,\R\ZZ_L)$ gives
	$$tb(\R\ZZ_K)+tb(\R\ZZ_L)\leq tb(\R\ZZ_K\cap\R\ZZ_L)+tb(\R\ZZ_K\cup\R\ZZ_L),$$
	which is equivalent to the statement of the lemma by~\eqref{Equ:DK-ZK}.
\end{proof}

     \begin{thm} \label{Thm-Main-4}
   If $K$ is a simplicial complex with vertex set $[m]$, then
$$
\widetilde{D}(K;\mathbb{F}) \leqslant g\left(m,\left[\frac{m-1}{3}\right]\right)+1
$$
for any field $\mathbb{F}$, where the equality holds if and only if 
$K=\Delta_{\left(\left[\frac{m-1}{3}\right]-1\right)}^{[m]}$.
\end{thm}
\begin{proof}
 Let
$\Sigma^{\widetilde{D}}(m)=\Big\{ K\in  \Sigma(m) \mid \widetilde{D}(K)= \underset{L\in\Sigma(m)}{\max} \widetilde{D}(L)\Big\}$, 
 which is a partially ordered set with respect to the inclusions of simplicial complexes.\n

	Suppose $K$ is a minimal or a maximal element of $\Sigma^{\widetilde{D}}(m)$. Then for every permutation $\phi$ of the vertex set of $K$, it follows from Lemma~\ref{lem:mv2} that 
	$$\widetilde{D}(K)+\widetilde{D}(\phi K) \leq \widetilde{D}(K \cap \phi K)+ \widetilde{D}(K \cup \phi K).$$
	This implies that $K=\phi K$. So $K$ must be
	 $\Delta^{[m]}_{(d)}$ where $d=\dim(K)$. 
	 Observe that any nonempty full subcomplex of $\Delta^{[m]}_{(d)}$ on $k$ vertices with $k\leq d+1$ is a simplex. So to compute $\widetilde{D}(\Delta^{[m]}_{(d)})$ via the formula~\eqref{Equ:D-tilde-K-J}, we only need to consider the full subcomplexes of $\Delta^{[m]}_{(d)}$ with more than $d+1$ vertices.
	Moreover, it is easy to see that the reduced homology
	 group of any nonempty full subcomplex of $\Delta^{[m]}_{(d)}$ always concentrates at degree $d$. Then an easy calculation shows that
	$$\widetilde{D}(\Delta^{[m]}_{(d)})=\sum_{i=0}^{m-d-2}\binom{m}{m-i}\binom{m-i-1}{d+1}+1 \overset{\eqref{Equ:g-def}}{=}g(m,d+1)+1.$$
	So the theorem follows from Lemma~\ref{lem:calc}
	in the Appendix.
\end{proof}

 \begin{rem}
  For a simplicial complex $K$ with vertex set $[m]$,
  the following combinatorial invariants of $K$ were studied by Codenotti, Spreer and Santos~\cite{Cod18}: 
    $$\tau_i(K)=\frac{1}{m+1}\sum\limits_{J\subset[m]}\frac{\wt{\beta}_i(K|_J)}{\binom{m}{|J|}}, \ i\geq -1. $$
    Formally, $\tau_i(K)$ is some sort of weighted average of the
    $i$-th Betti number of all the full subcomplexes of $K$, which has a similar flavor as our $\widetilde{D}(K)$.
   \end{rem}
  \vskip .4cm
  
 \section{Appendix}
\begin{lem}\label{lem:calc}
	For $0\leq d < m$, $g(m,d)=\sum_{j=d+1}^{m}\binom{m}{j}\binom{j-1}{d}$ reaches the maximum when and only when $d=\left[\frac{m-1}{3}\right]$.
\end{lem}
\begin{proof}
The cases $m\leq 4$ can be checked by hand, so we assume $m\geq5$ in the rest of the proof.
	It is easy to verify that 
	\begin{equation} \label{gmd-1}
	  g(m,d)+g(m,d-1)=2^{m-d}\binom{m}{d}.
	  \end{equation}
	So $g(m,d)-g(m,d-2)=2^{m-d}\binom{m}{d} - 2^{m-d+1}\binom{m}{d-1}$, which implies that
	$$g(m,d)-g(m,d-2) > 0 \, \Longleftrightarrow\, d<\frac{m+1}{3}.$$
	$$g(m,d)=g(m,d-2) \, \Longleftrightarrow\, d=\frac{m+1}{3}.$$
	Then we can deduce that
	\begin{itemize}
		\item if $m=3n$, then $g(m,d)$ is maximal only when $d=n-1$ or $n$;\n
		\item if $m=3n+1$, then $g(m,d)$ is maximal only when $d=n-1$ or $n$;\n
		\item if $m=3n+2$, then $g(m,d)$ is maximal only when $d=n-1, n$ or $n+1$ (since $g(n-1)=g(n+1)$ in this case).
	\end{itemize}
	So to prove the Lemma, we only need to prove: for any $n\geq 1$,
	\begin{itemize}
		\item[(a)] $g(3n,n-1)>g(3n,n)$;\n
		\item[(b)] $g(3n+1,n-1)<g(3n+1,n)$;\n
		\item[(c)] $g(3n+2,n-1)<g(3n+2,n)$.
	\end{itemize}
	\n
	 Note that by~\eqref{gmd-1},        
	  $$ g(m,d)>g(m,d-1)\, \Longleftrightarrow \, g(m,d-1)<2^{m-d-1}\binom{m}{d} \, \Longleftrightarrow \, 
	  \frac{g(m,d-1)}{\binom{m}{d}} < 2^{m-d-1}.$$	
	
	 We directly compute
	\begin{align} \label{Equ:g-m-d}
		\frac{g(m,d-1)}{\binom{m}{d}} =&\
		 \frac{1}{\binom{m}{d}} \sum_{j=d}^m \binom{m}{j}\binom{j-1}{d-1} 
		\stackrel{j\to m-j}{=} \frac{1}{\binom{m}{d}} \sum_{j=0}^{m-d} \binom{m}{m-j}\binom{m-j-1}{d-1} \notag \\
		=&\ \sum_{j=0}^{m-d}\frac{d}{m-j}\binom{m-d}{j}   =  
		 \int_0^1 d x^{d-1} \sum_{j=0}^{m-d}\binom{m-d}{j}x^{m-d-j}\mathrm{d}x \notag \\
		=&\ \int_0^1 d x^{d-1}(1+x)^{m-d}\mathrm{d}x. 
	\end{align}
	So we obtain
	\[\begin{aligned}
		g(m,d)>g(m,d-1) 
		&\Longleftrightarrow \,\int_0^1dx^{d-1}(1+x)^{m-d}\mathrm{d}x<2^{m-d-1}.
	\end{aligned}\]
	 Moreover, by the Cauchy-Schwarz inequality:
	\begin{align*}
	 \int_0^1 d x^{d-1}(1+x)^{m-d}\mathrm{d}x &\leq
	d\, \sqrt{\int_0^1 (x^{d-1})^2 \,\mathrm{d}x} \, \sqrt{\int_0^1 [(1+x)^{m-d}]^2\,\mathrm{d}x } \\
	 &\leq 2^{m-d-1} \sqrt{\frac{8d^2}{(2d-1)(2m-2d+1)}} .
	 \end{align*}
	So to prove (b) and (c), we only need to show $8d^2<(2d-1)(2m-2d+1)$ for $(m,d)=(3n+1,n)$ and $(m,d)=(3n+2,n)$, which is easy to check.	
	Similarly, 
	\[\begin{aligned}
		g(m,d)<g(m,d-1) & \, \Longleftrightarrow \, g(m,d)<2^{m-d-1}\binom{m}{d} \, \Longleftrightarrow \,\frac{g(m,d)}{\binom{m}{d}} < 2^{m-d-1} \\
		&\Longleftrightarrow \,\int_0^1(m-d)x^{d}(1+x)^{m-d-1}\mathrm{d}x<2^{m-d-1},
	\end{aligned}\]
    where we use the result of~\eqref{Equ:g-m-d}
     with $d$ substituted by $d+1$ and the relation $\binom{m}{d+1}=\frac{m-d}{d+1}\binom{m}{d}$. Then again by the Cauchy-Schwarz inequality, we obtain
    $$\int_0^1(m-d)x^{d}(1+x)^{m-d-1}\mathrm{d}x\leq 2^{m-d-1} \sqrt{\frac{2(m-d)^2}{(2d+1)(2m-2d-1)}}.$$
    So to prove (a), we only need to show $2(m-d)^2<(2d+1)(2m-2d-1)$ for $(m,d)=(3n,n)$, which is also easy to check.	
	This finishes the proof.
\end{proof}


\vskip .8cm
 
  
\begin{thebibliography}{99}



\bibitem{Abb01}
  A.~Abbondandolo, \emph{Morse theory for Hamiltonian systems}, Chapman $\&$ Hall/CRC Res. Notes Math., 425, 2001.
  
\bibitem{AllPuppe93}
   C.~Allday and V.~Puppe, \emph{Cohomological methods in transformation groups}, Cambridge Studies in
Advanced Mathematics 32, Cambridge University Press, 1993.


\bibitem{And02}
		I.~Anderson,  \emph{Combinatorics of finite sets}, Dover Publications, Inc., Mineola, NY, 2002.
		

  \bibitem{BjrnKal88}
   A.~Bj\"{o}rner and G.~Kalai, An extended Euler-Poincar\'e theorem, Acta Math. 161 (1988), no.~\textbf{3-4}, 279--303.
  

 \bibitem{BP15}
V.~M.~Buchstaber and T.~E.~Panov, \emph{Toric Topology}.
 Mathematical Surveys and Monographs, Vol.~204, 
 American Mathematical Society, Providence, RI, 2015.
 
 \bibitem{CaoLu12}
 X.~Cao and Z.~L\"u,
\emph{M\"obius transform, moment-angle complexes and Halperin-Carlsson conjecture}, J. Algebraic Combin.35 (2012), no.~\textbf{1}, 121--140.

\bibitem{Carl86}
 G.~Carlsson, \emph{Free $(\Z/2)^k$ actions and a problem in commutative algebra}, Lecture Notes in Math., \textbf{1217}, Springer, Berlin, 1986, 79--83.

\bibitem{Cod18}
		 G.~Codenotti, J.~Spreer and F.~Santos, \textit{Average Betti numbers of induced subcomplexes in triangulations of manifolds}, Electron. J. Combin.~27 (2020), no.\textbf{3}, Paper No. 3.40.

     
\bibitem{DaJan91}  M.~W.~Davis and T.~Januszkiewicz, \textit{Convex polytopes,
Coxeter orbifolds and torus actions}. Duke Math. J. \textbf{62}
(1991), no.~\textbf{2}, 417--451. 


\bibitem{ErKoRa61}
  P.~Erd\"os, C.~Ko and R.~Rado, \emph{Intersection theorems for systems of finite sets}, Quart. J. Math. Oxford Ser. (2)  
  12 (1961), 313--320.
  
  \bibitem{Gri82}
		J.~Griggs, \emph{Collections of subsets with the Sperner property}, Trans. Amer. Math. Soc. 269 (1982), no.~\textbf{2}, 575--591.
		
\bibitem{Halp85}
  S.~Halperin, \emph{Rational homotopy and torus actions}, In: Aspects of Topology. London
Math. Soc. Lecture Note Series, 93. Cambridge Univ. Press, Cambridge (1985), 293--306.
 
\bibitem{Hat02}
 A.~Hatcher, Algebraic topology, Cambridge University Press, Cambridge, 2002.
 
\bibitem{Hoc77}
 M.~Hochster, \emph{Cohen-Macaulay rings, combinatorics, and simplicial complexes}, Ring Theory II
(Proc. Second Oklahoma Conference), B. R. McDonald and R. Morris, eds. Dekker, New York (1977), 171--223. 

\bibitem{Kal84}
 G.~Kalai, \emph{Characterization of $f$-vectors of families of convex sets in $R^d$, I. Necessity of Eckhoff's conditions}, Israel J. Math. 48 (1984), no.~\textbf{2-3}, 175--195. 
 
\bibitem{Kal85}
 G.~Kalai, \emph{$f$-vectors of acyclic complexes}, Discrete Math.~55 (1985), no.~\textbf{1}, 97--99. 
     
\bibitem{Lih80}
		K.~W.~Lih, \emph{Sperner families over a subset}, 
		J. Combin. Theory Ser. A 29 (1980), no.~\textbf{2}, 182--185.
	     
 
\bibitem{Spern28}
 E.~Sperner, \emph{Ein Satz \"uber Untermenge einer endlichen Menge}, Math. Z., 27 (1928), 544--548.


\bibitem{Stanley07}
   R.~Stanley, \emph{Combinatorics and commutative algebra}, 2nd edition. Birkh\"auser Boston, 2007.
 
 
\bibitem{Uto11}
  Y.~M.~Ustinovsky, \textit{Toral rank conjecture for moment-angle complexes},
   Mathematical Notes \textbf{90} (2011), no.~\textbf{2}, 279--283.
   
   
\bibitem{YuMas21}
L.~Yu and M.~Masuda, \emph{On descriptions of products of simplices}, Chinese Ann. Math. Ser. B 42 (2021), no.~\textbf{5}, 777--790.   
   
\end{thebibliography}
\end{document}